\documentclass[12pt]{amsart}
\usepackage{amsmath,amssymb,amsthm,amscd}

\def\ri{\right}

\def\p{\partial}

\newcommand\C{{\mathbb C}}

\def\R{R_\abb}

\def\be{\begin{equation}}
\def\ee{\end{equation}}

\def\ri{\right}

\def\p{\partial}

\def\C{\Bbb C}

\def\p{\partial}

\def\p{\partial}

\def\C{\Bbb C}
\def\R{\Bbb R}

\newtheorem{thm}{Theorem}[section]

\newtheorem{lem}{Lemma}[section]

\theoremstyle{definition}
\newtheorem{defn}{Definition}[section]
\theoremstyle{remark}

\newtheorem{rem}{Remark}[section]

\begin{document}

\title[Lagrangian self-similar solitons]
{Entire self-similar solutions to Lagrangian Mean curvature flow}

\author{Albert CHAU}
\address{Department of Mathematics\\
University of British Columbia\\
Vancouver, B.C., V6T 1Z2\\
Canada}
\email{chau@math.ubc.ca}

\author{Jingyi CHEN}
\email{jychen@math.ubc.ca}

\author{Weiyong He}
\email{whe@math.ubc.ca}
\thanks{The first two authors are partially supported by NSERC, and the third author is
partially supported by a PIMS postdoctoral fellowship.}
\date{May 19, 2009}

\begin{abstract}
We consider self-similar solutions to mean curvature evolution of entire Lagrangian graphs.  When the Hessian of the potential function $u$ has eigenvalues strictly uniformly between $-1$ and $1$,  we show that on the potential level all the shrinking solitons are quadratic polynomials while the expanding solitons are in one-to-one correspondence to functions of homogenous of degree 2 with the Hessian bound. We also show that if the initial potential function is cone-like at infinity then the scaled flow converges to an expanding soliton as time goes to infinity.  
\end{abstract}

\maketitle

\section{introduction}
Let $u: \R^n\rightarrow \R$ be a smooth function. Then $L_u=\{(x, Du(x)):x\in\R^n\}$ defines a graphical Lagrangian submanifold in $\R^{2n}=\R^{n}\oplus \sqrt{-1}\,\R^n=\C^n.$
On the other hand, any entire graphical Lagrangian submanifold in $\R^{2n}$ can be obtained by such a potential $u$, up to addition of  constants. Let 
\[
 G(D^2u)=\dfrac{1}{\sqrt{-1}} \log \dfrac{\det (I_n+ \sqrt{-1}D^2u)}{\sqrt{\det (I_n+ (D^2u)^2)}},
 \] 
 where $I_n$ is the $n$ dimensional identity matrix.  The operator $G(D^2u)$ is strictly elliptic and arises in the equation for minimal Lagrangian graphs.  Namely, $L_u$ is a minimal submanifold in 
 $\R^{2n}$ exactly when $G(D^2u)$ is equal to some constant $\Theta$.  In this note, we will consider the following two equations for $u(x)$:

\begin{equation}\label{expander-elliptic}
G(D^2u)-u+\frac{1}{2}\nabla u\cdot x=0,
\end{equation}

\begin{equation}\label{shrinker-elliptic}
G(D^2u)+u-\frac{1}{2}\nabla u\cdot x=0.
\end{equation}

Our motivation is that \eqref{expander-elliptic} and \eqref{shrinker-elliptic} are the defining equations for self-expanding and self-shrinking solutions to the Lagrangian mean curvature flow of entire graphs in $\R^{2n}$.   We describe this in more detail in the next section.  We will refer to an entire solution to \eqref{expander-elliptic} as a \begin{it}self-expanding soliton\end{it} and an entire solution to \eqref{shrinker-elliptic} as a \begin{it}self-shrinking soliton. \end{it}

\begin{defn} We say that $u_0$ satisfies  {\it Condition A} if $D^2u_0\in L^\infty(\R^n)$ and
$$-(1-\delta) I_n\leq  D^2u_0 \leq (1-\delta) I_n$$
for some $\delta \in (0, 1)$ where $I_n$ is the $n$ dimensional identity matrix.  
We say $u_0$ satisfies {\it Condition B} if  $D^2u_0\in L^\infty(\R^n)$ and
$$u_0(x)=\dfrac{1}{\lambda^2} u_0 (\lambda x)$$ for any $\lambda>0.$
\end{defn}

We now state our main result
\begin{thm}\label{soliton}  
If $u$ is a smooth self-shrinking soliton satisfying Condition A,  then $u$ is a quadratic. There exists a one-to-one correspondence between smooth self-expanding solitons satisfying  Condition A and functions which satisfy Condition A and Condition B. In particular, 
there are infinitely many nontrivial smooth self-expanding solitons satisfying  Condition A.  
\end{thm}

We also obtain a convergence result for the following fully nonlinear parabolic equation which arises from the Lagrangian mean curvature flow (see \S2):
\begin{equation}\label{PMA}
\left\{%
\begin{array}{ll}
 &\dfrac{du}{dt} = G(D^2u)\bigskip\\
 &u(x, 0)=u_0(x).
\end{array}%
\right.
 \end{equation}

\begin{thm}\label{T-2}
Let $u_0$ satisfy Condition A and suppose that  
\begin{equation}\label{hypothesis}
\lim_{\lambda\rightarrow \infty}\lambda^{-2}u_0(\lambda x)\rightarrow U_0(x)
\end{equation}
 for some $U_0(x)$ as $\lambda\rightarrow \infty$. Let $u(x, t)$ and $U(x, t)$ be solutions to \eqref{PMA} with initial data $u_0(x)$ and $U_0(x)$ respectively.
Then $t^{-1}u(\sqrt{t}x, t)$ converges to $U(x, 1)$ uniformly and smoothly in compact subsets of $\R^n$ when $t\rightarrow \infty$, where $U(x, 1)$ is a self-expanding soltion.
\end{thm}

%\begin{rem} 
If we define $v(x, s)=t^{-1}u(\sqrt{t}x, t)$ where $s=\log t$ and $t\geq 1.$ Then $v(x, s)$ satisfies the following equation
\[
\frac{\p v}{\p s}=G(D^2 v)-v+\frac{1}{2}\nabla v\cdot x.
\]
Theorem \ref{T-2} implies that $v(x, s)$ converges to $U(x, 1)$ when $s\rightarrow \infty$. 
%\end{rem}
For mean curvature flow of entire graphic hypersurface, if the initial surface is cone-like at infinity, in the sense that the normal component of the graph decays to zero at infinity at a certain rate, then the scaled mean curvature flow converges to a self-expanding soliton \cite{EH}, where the monotonicity formula plays an important role. Our method is different  from \cite{EH} and is based on the existence result and the estimates in \cite{CCH} and the uniqueness result in \cite{chen-pang}. 

 Some interesting examples of non-graphical self-similar solutions to Lagrangian mean curvature flow in $\C^n$ have recently been constructed \cite{AH, JLT}; also see references therein.  
 
 %We refer to the references in there for results on self-similar solutions to the general mean curvature flow.

\section{Preliminaries}
We now describe the connections among the elliptic equations \eqref{expander-elliptic} and \eqref{shrinker-elliptic}, the parabolic equation \eqref{PMA} and the Lagrangian mean curvature flow. 

 Let $F_0: M^n\rightarrow \mathbb{R}^{n+m}$ be an embedding of a manifold $M^n$ into  $\mathbb{R}^{n+m}$.  Then the mean curvature flow with initial condition $F_0$ is the equation

\begin{equation}\label{LMCF}
\left\{
\begin{array}{ll}
 &\dfrac{dF}{dt} =  H \bigskip\\ 
 &F(x, 0)=F_0(x)
 \end{array}
 \right.
\end{equation}
where $H(x, t)$ is the mean curvature vector of the submanifold $F(\cdot, t)(M)$ of $\R^{2n}$
at $F(x, t)$.  The solution $F(\cdot, t): M^n\rightarrow \mathbb{R}^{n+m}$ is called {\it self-expanding} if it is defined for all $t>0$ and $F(\cdot, t)$ has the form
\begin{equation}\label{expander}
M_t=\sqrt{t} \,  M_1,\hspace{12pt}\hbox{ for all $ t>0$}
\end{equation}
where $M_t=F(\cdot, t)(M)$ and $\sqrt{t}$ represents the homothety $p\mapsto \sqrt{t} p$, $p\in \mathbb{R}^{n+m}$.  Similarly, $F(\cdot, t): M^n\rightarrow \mathbb{R}^{n+m}$  is called {\it self-shrinking}  if it is defined for all $t<0$ and $F(x, t)$ has the form
\begin{equation}\label{shrinker}
M_{t}=\sqrt{-t}\,M_{-1},\hspace{12pt}\hbox{for all $ t<0$}.
\end{equation}

Now when $F_0(x)=(x, Du_0(x))$ for $x\in \R^{n}$ defines an entire Lagrangian graph for some potential $u_0:\R^n \to \R$, then \eqref{LMCF} is equivalent to the following fully nonlinear parabolic equation \eqref{PMA}:

\begin{equation}\nonumber
\left\{%
\begin{array}{ll}
 &\dfrac{du}{dt} = G(D^2u)\bigskip\\
 &u(x, 0)=u_0(x).
\end{array}%
\right.
 \end{equation} 
 In particular, if $u$ is a solution to \eqref{PMA} then there exists a family of diffeomorphisms \cite{smoczyk} $r_t:\R^n \to \R^n$ such that $F(x, t)=(r_t(x), Du(r_t(x), t))$ solves \eqref{LMCF}.    
Now suppose the family $F(x, t)$ satisfies \eqref{expander}.  Then we have
\begin{equation}\label{u-expander}
D\left(u(x, t)-tu\left(\frac{x}{\sqrt{t}}, 1\ri)\right)=0, \,\,\forall  ~t>0.
\end{equation} Thus by letting $t=1$, we have
\begin{equation}\label{u-expander-1}
u(x, t)=tu\left(\frac{x}{\sqrt{t}}, 1\right), \,\, \forall ~t>0.
\end{equation} 
Using \eqref{PMA} and \eqref{u-expander-1} we verify directly that $u(x, 1)$ satisfies \eqref{expander-elliptic}, in other words, $u(x, 1)$ is a self-expanding soliton.  Similarly, we can derive that if $F(x, t)$ is a self-shrinking solution to \eqref{LMCF} with corresponding solution $u(x, t)$ to \eqref{PMA}, then $u(x, t)$ satisfies

\begin{equation}\label{u-shrinker-1}
u(x, t)=-t u\left(\frac{x}{\sqrt{-t}}, -1\right),  \forall ~t>0,
\end{equation}
and $u(x, -1)$ satisfies \eqref{shrinker-elliptic}.  So $u(x, 1)$ is a self-shrinking soliton.

Conversely, it is not hard to show that if $u(x)$ solves either \eqref{expander-elliptic}  or \eqref{shrinker-elliptic}, then using \eqref{u-expander-1} or \eqref{u-shrinker-1} respectively we can generate a solution $F(x, t)$ to \eqref{LMCF} which is either shrinking or expanding.  We illustrate this in detail in the expanding case.  The case for shrinking solutions is similar.  Suppose $u(x)$ solves \eqref{expander-elliptic}  and define $u(x, t):=tu(x/\sqrt{t})$.  Then as above, the family  $M_t:=\{(x, Du(x, t)):x\in \R^n\}$ is easily checked to satisfy \eqref{expander}.  On the other hand,  we also have
\begin{equation}
\begin{split}
\dfrac{du}{dt} (x, t)=& \,\,  u\left(\frac{x}{\sqrt{t}}\right) -\frac{1}{2} \nabla u({x}{\sqrt{t}})\cdot \frac{x}{\sqrt{t}}\\
=&\,\,G\left(D_y^2 u\left(\frac{x}{\sqrt{t}}\right)\right)\\
=&\,\,G\left(D^2 u(x, t)\right)
\end{split}
\end{equation}
where $y=x/\sqrt{t}$.  In other words, $u(x, t)$ solves \eqref{PMA} and by the discussion above there exists a family $r_t$ such that $F(x, t)=(r_t(x), Du(r_t(x), t))$ is solution to \eqref{LMCF} which is self-expanding.

\section{Proofs of Theorems}

The proof of Theorem \ref{soliton} is to consider the parabolic equation \eqref{PMA} with Lipschitz initial data. Let us recall the main theorem in \cite{CCH}:
\begin{thm}\label{entire}
Let $u_0:\R^n \to \R$ satisfy Condition A.  Then
\eqref{PMA} has a longtime smooth solution $u(x, t)$ for all $t>0$ with initial condition $u_0$ such that the following estimates hold:

(i)  $-(1-\delta) I_n\leq D^2u\leq (1-\delta) I_n$  for all $t>0$,

(ii) $\sup_{x\in \R^n}\left|D^l u(x,t)\right|^2 \leq C(l, \delta)/t^{l-2}$,  for  $l\geq3$ and some $C(l, \delta)$.

(iii) $u(x, t)\in C^{\infty}(\R^n \times (0, \infty))$ and $u(x, t)$ converges to $u_0(x)$ in the Lipschitz norm when $t\rightarrow 0$. 
\end{thm}

We will apply Theorem \ref{entire} and the uniqueness result in \cite{chen-pang} for  the parabolic equation \eqref{PMA} to prove Theorem \ref{soliton}. The proof consists of the following two lemmas.

\begin{lem}\label{T-3}
Let $u_0:\R^n \to \R$  satisfy Condition A and  Condition B.
Then \eqref{PMA} has a  unique longtime smooth solution $u(x, t)$ for all $t>0$ with initial condition $u_0$ such that $u(x, t)=t u(x/\sqrt{t}, 1)$. 
In particular $u(x, 1)$ is a smooth self-expanding soliton satisfying \eqref{expander-elliptic}.  Conversely, if $v$ is smooth a self-expanding soliton satisfying Condition A, then there exists $u_0$ satisfying  Condition A and Condition B such that $u(x, t)$ is the unique solution of \eqref{PMA} with initial condition $u_0$ and $v=u(x, 1)$.
\end{lem}

\begin{proof}
If $u_0$ satisfies Condition A, then by Theorem \ref{entire}, there exists a smooth solution $u(x, t)$ to \eqref{PMA} for all $t>0$ with initial data $u_0$. It is clear that 
\[
u_\lambda(x, t):=\lambda^{-2}u(\lambda x, \lambda^2 t)
\]
is also a solution  to \eqref{PMA} with initial data \[u_\lambda(x, 0)=\lambda^{-2}u_0(\lambda x)=u_0(x)
\]
where we have used that $u_0$ satisfies Condition B. 
Since $u_\lambda(x, 0)=u_0$, the uniqueness result in \cite{chen-pang} implies
\[
u(x, t)=u_\lambda(x, t).
\]
 for any $\lambda>0$.  Therefore $u(x, t)$ satisfies \eqref{u-expander-1}, and hence $u(x, 1)$ solves \eqref{expander-elliptic}.  In other words, $u(x, 1)$ is a smooth self-expanding soliton.

Now suppose that $v$ is a smooth solution to \eqref{expander-elliptic} satisfying Condition A (in fact, it suffices to assume $D^2 v$ is bounded). 
Define $u(x, t)$ for $t>0$ by
\[
u(x, t)=tv\left(\frac{x}{\sqrt{t}}\right).
\]
It is clear that $u(x, t)$ satisfies the evolution equation in \eqref{PMA} since $v$ satisfies \eqref{expander-elliptic}. Now we claim that the limit of $u(x, t)$ exists when $t$ goes to zero. 
To see this, begin by noting $u(0, t)=tv(0)$ for $t>0$ and so $$\lim_{t\rightarrow 0}u(0, t)=0.$$  Moreover, it is clear that $$Du(0, t)=\sqrt{t}Dv(0)$$  and that for any $t>0$
\[
-(1-\delta) I_n\leq  D^2 u(x, t) \leq (1-\delta) I_n
\]
since 
\[
D^2u (x, t)=D^2 \left(tv\left(\frac{x}{\sqrt{t}}\right)\right)=D^2 v\left(\frac{x}{\sqrt{t}}\right).
\]
We may then conclude that for any  sequence $t_i\rightarrow 0$ there is a subsequence $t_{k_i}$ such that $u(x, t_{k_i})$ converges in $C^{1, \alpha}$ uniformly in compact subsets of $\R^n$ for any $0<\alpha<1$.  This limit is in fact independent of the sequence $t_i$.  Indeed, let $u_1$ and $u_2$ be two such limits along sequences $t_{i}$ and $t'_{i}$ respectively.  Since $u(x, t)$ is a solution to \eqref{PMA}, $\p u /\p t$ is also uniformly bounded for any $t>0$, $x\in \R^n$. Thus for any $x\in \R^n$ we may have 
$$|u(x, t_i) - u(x, t'_i)\left|\leq C|t_i - t'_i\ri|$$ for some $C$ independent of $i$.  Letting $i \to \infty$, we conclude that $u_1(x)=u_2(x)$.  So for different sequences $t_i\rightarrow 0$, the limit is unique. Let 
$$
u_0(x)=\lim_{t\rightarrow 0}u(x, t).
$$
Then  $D^2u_0\in L^\infty$ and we have 
\[
-(1-\delta) I_n\leq  D^2 u_0\leq (1-\delta) I_n.
\]
Further, 
\begin{eqnarray*}
\frac{1}{\lambda^2}u_0(\lambda x) &=& \frac{1}{\lambda^2}\lim_{t\to 0} tv\left(\frac{\lambda x}{\sqrt{t}}\right)\\
&=&\lim_{t\to 0}\left( \frac{\sqrt{t}}{\lambda}\right)^2 v\left(\frac{\lambda x}{\sqrt{t}}\right)\\
&=& u_0(x).
\end{eqnarray*}
Therefore $u_0$ satisfies Condition B. \end{proof}

\begin{rem}Lemma \ref{T-3} provides many examples of  entire graphical Lagrangian self-expanding solitons. For instance, we can choose $u_0$ such that
\[
u_0(x)=\left\{\begin{array}{ll}&\displaystyle{a_1 x_1^2+\sum_{i=2}^n a_i x_i^2,} \,\,\,\,\,\,\,\,\,\,\,\,\mbox{if $x_1\geq 0$}\medskip\\
&\displaystyle{-a_1x_1^2+\sum_{i=2}^n a_i x_i^2,} \,\,\,\,\,\,\,\mbox{if $x_1<0$}.
\end{array}
\right.
\]
If  $a_i\in (-1+\delta, 1-\delta)$ for any fixed $\delta\in(0,1)$, then $u_0$ satisfies  Condition A and Condition B and $u(x, 1)$  is a self-expanding soliton which satifies \eqref{expander-elliptic}. As $u_0$ is not smooth, $u(x,1)$ cannot be a 
quadratic polynomial. 
\end{rem}

Next, we show that smooth self-shrinking solitons are trivial if Condition A holds.

\begin{lem}If $v$ is a smooth solution of \eqref{shrinker-elliptic} such that $-(1-\delta)I_n\leq D^2v\leq (1-\delta)I_n$, then $v$ is a quadratic polynomial.
\end{lem}

\begin{proof}If $v$ is a smooth solution to \eqref{shrinker-elliptic}, then 
$$
u(x, t)=(1-t)v\left(\frac{x}{\sqrt{1-t}}\right)
$$ 
is  a solution to \eqref{PMA} for $t\in (0, 1)$ and $u(x, 0)=v(x)$, hence Theorem \ref{entire} applies to $u(x,t)$ since this solution is unique, and in particular $\left|D^3u(x, t)\ri|\leq C$ for some constant $C$ when $t\geq 1/2$ for any $x$.  But one checks directly that 
$$
D^3u(x, t)=\frac{1}{\sqrt{1-t}}D^3v\left(\frac{x}{\sqrt{1-t}}\right).
$$ 
This implies $$\left|D^3v(x)\ri|=\left|D^3v\left(\dfrac{x\sqrt{1-t}}{\sqrt{1-t}}\ri)\ri|\leq C\sqrt{1-t}$$ for any $x$.  It follows that $D^3v(x)=0$, thus $v$ is quadratic.
\end{proof}

Now we prove Theorem \ref{T-2}.
\begin{proof}Let $u(x, t)$ be the solution to \eqref{PMA} with initial data $u_0$ satisfying Condition A. It is clear that for any $\lambda$, 
$$
u_\lambda(x, t)=\lambda^{-2}u(\lambda x, \lambda^2 t)
$$ 
is a solution of \eqref{PMA} with initial data $u_\lambda(x, 0)=\lambda^{-2}u_0(\lambda x)$ satisfying Condition A.   By the uniqueness result in \cite{chen-pang}, we may then apply the estimates in Theorem \ref{entire} to $u_\lambda(x, t)$.  For any sequence $\lambda_i\rightarrow \infty$, consider the solutions $u_{\lambda_i}(x, t)$. 
For $t>0$, it is clear that $$D^2u_{\lambda_i}(x, t)=D^2 u(\lambda_i x, \lambda_i^2 t)$$ and so by Theorem \ref{entire},
\[
 -(1-\delta)I_n\leq D^2 u_{\lambda_i}(x, t)\leq (1-\delta)I_n
\]
 for all $x$ and $t\geq0$, and for $t>0$ and $l\geq 3$ 
\[
\left|D^l u_{\lambda_i}(x, t)\right|\leq C(l, \delta)\sqrt{t}^{2-l}.
\]  Thus by \eqref{PMA} we can also get the following estimates that, for any $m\geq 1, l\geq 0 $, there is a constant $C(m, l, \delta)$ such that
\[
\left|\frac{\p^m}{\p t^m} D^l u_{\lambda_i}\right|\leq C(m, l, \delta) \sqrt{t}^{2-l-2m}.
\]
In particular, there are constants $C$ and $ C(\delta)$ such that 
 \[
\left|\frac{\p u_{\lambda_i}}{\p t}\right|\leq C, \,\,\,\,\left|\frac{\p Du_{\lambda_i}}{\p t}\right|\leq \frac{C(\delta)}{\sqrt{t}}
\]
for all $t>0$.  We observe that $$u_{\lambda_i}(0, 0)=\lambda_i^{-2}u_0(0)$$ and  $$Du_{\lambda_i}(0, 0)=\lambda_i^{-1} Du_0(0)$$ are both bounded, thus $u_{\lambda_i}(0, t)$ and $Du_{\lambda_i}(0, t)$ are uniformly bounded in $i$ for any fixed $t$. 
By Arzel\`a-Ascoli theorem, there exists a subsequence $\lambda_{k_i}$ such that $u_{\lambda_{k_i}}(x, t)$ converges smoothly and uniformly in compact subsets of $\R^n\times (0, \infty)$ to  a solution $U(x, t)$ of \eqref{PMA}. 
Moreover, $U(x, t)$ satisfies the estimates in Theorem \ref{entire}.  Since $\p U/\p t$ is uniformly bounded for any $t>0$, $U(x, t)$ converges to some function $U(x, 0)$  when $t\rightarrow 0$. In particular  we have
\[
U(x, 0)=\lim_{\lambda_{k_i}\rightarrow \infty} \lambda_{k_i}^{-2}u_0(\lambda_{k_i}x).
\]
It is clear that $U(x, 0)$ satisfies Condition A.
Similarly as in the proof of Lemma 2.1, the hypothesis  (\ref{hypothesis}) implies that  $U(x, 0)$ satisfies Condition B.  Thus by Theorem \ref{soliton} we know that $U(x, 1)$ is a self-expanding soliton, and $U(x, t)$ satisfies \eqref{u-expander-1}.  

Moreover, notice that $U(x,0)=U_0(x)$ by (\ref{hypothesis}) and $U_0(x)$ does not depend on the sequence $\lambda_i$, it then follows from the uniqueness result in \cite{chen-pang} that the solution $U(x, t)$ is also independent of the sequence $\lambda_i$ and we may write 
\[
u_\lambda(x, t)\rightarrow U(x, t).
\]
as $\lambda\rightarrow \infty$.  In particular, letting $\lambda=\sqrt{t}$ we have $t^{-1}u(\sqrt{t}x, t)$ converges to $U(x, 1)$ smoothly and uniformly in compact subsets of $\R^n$ when $t\rightarrow \infty$.
\end{proof}

\begin{rem}
 The family of Lagrangian graphs $(x + at, Du_0(x)+bt)$ is a translating solution to \eqref{LMCF} precisely when $u$ satisfies 
\begin{equation}\label{E-soliton-1}
\sum_i\arctan \lambda_i+\sum_i a_i\frac{\partial u_0}{\partial x_i}-\sum_ib_ix_i=c,
\end{equation}
for some constant $c$.  In \cite{CCH} the authors proved that if $u$ is a smooth solution to \eqref{E-soliton-1} which satisfies condition A, then $u$ is a quadratic function.  We observe here that this result can also be obtained  from the uniqueness result in \cite{chen-pang} together with Theorem \ref{entire}.  Namely, if $u_0(x)$ solves \eqref{E-soliton-1} then $v(x, t)= u_0(x-a t)+b\cdot x \,t +c t$ solves \eqref{PMA} with initial condition $v(x,0)=u_0(x)$.  On the other hand, Theorem  \ref{entire} guarantees a longtime solution $u(x, t)$ to \eqref{PMA} with initial condition $u_0$ for which $\sup_{x\in\R^n}|D^3u(x, t)|\to 0$ as $t\to \infty$.  Then as  $u(x, t)=v(x, t)$ by the uniqueness result in  \cite{chen-pang}, $\sup_{x\in \R^n}|D^3u_0(x-at)|=\sup_{x\in \R^n}|D^3v(x, t)|\to 0$ as $t\to \infty$ and we conclude that $\sup_{x\in\R^n}|D^3u_0(x)|=0$, in other words $u_0$ must be quadratic. \end{rem}

\bibliographystyle{amsplain}

\end{document}